\documentclass[a4paper,11pt]{article}
\usepackage{amsfonts,graphics,amsmath,amsthm,amsfonts,amscd,amssymb}
\usepackage{amsmath,latexsym,multicol,mathrsfs,stmaryrd,color}
\usepackage{epsfig,url,pxfonts,accents}
\usepackage{flafter,multirow}
\usepackage{galois,enumerate}
\usepackage{fancyhdr}
\usepackage{hyperref}
\usepackage{textcomp}

\hypersetup{colorlinks=true, linkcolor=black}
\newcommand{\ovline}[1]{%
  \vbox{%
    \hrule height 0.6pt%            % Line above with certain width
    \kern0.22ex%                    % Distance between line and content
    \hbox{%
      \kern-0.05em%                  % Distance between content and left side of box, negative values for lines shorter than content
      \ifmmode#1\else\ensuremath{#1}\fi%  % The content, typeset in dependence of mode
      \kern0em%                  % Distance between content and left side of box, negative values for lines shorter than content
    }% end of hbox
  }% end of vbox
}

\makeatletter

\def\jobis#1{FF\fi
  \def\predicate{#1}%
  \edef\predicate{\expandafter\strip@prefix\meaning\predicate}%
  \edef\job{\jobname}%
  \ifx\job\predicate
}

\makeatother

\if\jobis{proposal}%
\else
\fi

\usepackage[all]{xy}

 \numberwithin{equation}{subsection}
 \numberwithin{footnote}{subsection}

 \newtheorem{thm}[subsection]{Theorem}

 \newtheorem{quest}[subsection]{Question}
{%_%_%_%_% upright style; roman (non-italic) text
    \newtheoremstyle{upright}%
        {8pt plus2pt minus4pt}%
        {8pt plus2pt minus4pt}%
        {\upshape}%
        {}%
        {\bfseries\scshape}%
        {}%
        {1em}%
        {}%
\theoremstyle{upright}

 \newtheorem{exam}[subsection]{Example}

}

\newcommand{\cl}{\mathcal}

\newcommand{\oo}{\mathscr O}

\newcommand{\BP}{\mathbb P}

\newcommand{\Z}{\mathbb Z}

\newcommand{\gl}{\mathrm{GL}}

\newcommand{\bir}{\mathrm{Bir}}

\begin{document}
\title{A remark on the rank of finite $p$-groups of birational automorphisms}
\author{Jinsong Xu}
\date{}
\maketitle
\abstract{In this short note, we improve a result of Prokhorov and Shramov on the rank of finite $p$-subgroups of the birational automorphism group of a rationally connected variety. Known examples show that they are sharp in many cases.}

\section{Introduction}

\noindent We work over an algebraically closed field of characteristic $0$.

In \cite[\S 6]{s}, Serre asked if the following property holds (cf. \cite[Question 1.1]{ps4}):

\begin{quest} \label{serre}
Is there a constant $L(n)$ (depending only on $n$) such that if a prime numebr $p$ is greater than $L(n)$, then any finite $p$-subgroup $G$ of the Cremona group $\mathrm{Cr}_n(k)$ is abelian, and can be generated by at most $n$ elements?
\end{quest}

The question was answered affirmatively by Prokhorov and Shramov in their study of Jordan property of the Cremona groups \cite{ps1} under the assumption of the well-known conjecture (now a theorem, see \cite{bi}) of boundedness of Fano varieties with mild singularities. As a direct consequence of their \emph{almost fixed point property} for finite subgroup actions on rationally connected varieties, they showed more generally that

\begin{thm} \cite[Theorem 1.10]{ps1}
	There is a constant $L = L(n)$ such that for any rationally connected variety $X$ of dimension $n$ defined over an arbitrary (not necessarily algebraically closed) field $k$ of characteristic $0$ and for any prime $p > L$, every finite $p$-subgroup $G \subset \bir(X)$ is an abelian group generated by at most $n$ elements.
\end{thm}

It is interesting to find effective values of the constants $L(n)$. When $n \leq 3$, they were studied a few years earlier for elementary abelian $p$-groups, i.e., $G \simeq (\Z/p\Z)^r$ for some $r$, see \cite{be} and \cite{p2}.

Later, using explicit algebraic geometry of surface and threefolds, it was shown that

\begin{thm} \label{surface}(\cite[Proposition 1.7]{ps4})
	Let $S$ be a rational surface, and let $G \subset \mathrm{Cr}_2(k)$ be a finite $p$-subgroup. Then $G$ is an abelian group generated by at most two elements if $p \geq 5$.
\end{thm}

The lower bound $p \geq 5$ in the above theorem is sharp due to existing examples.

\begin{thm} (\cite[Theorem 1.2.4, Corollary 1.2.5]{ps3})
	Let $X$ be a rationally connected threefold, and let $G \subset \bir(X)$ be a finite $p$-group. Then $G$ is an abelian group generated by at most three elements if $p > 10 368$.
\end{thm}

The lower bound in this case was then substantially refined in \cite{ps4}:
\begin{thm} \cite[Theorem 1.5]{ps4})
	Let $X$ be a rationally connected threefold and let $G \subset \bir(X)$ be a $p$-subgroup. Then $G$ is an abelian group generated by at most three elements if $p \geq 17$.
\end{thm}

The goal of this note is to answer Question \ref{serre} in every dimension without using boundedness of Fano varieties, and to further improve effective values of $L(n)$. Indeed, we will show that $L(n) \leq n+1$.

\section{Main theorem}
Our main tool is a remarkable fixed point theorem of Olivier Haution. In this theorem, the field $k$ is algebraically closed of arbitrary characteristic. 

\begin{thm} \label{fixed} \cite[Theorem 1.2.1]{h}
	Let $X$ be a projective variety with an action of a finite $p$-group $G$. Assume that one of the following conditions holds:
	
	(i) $G$ is cyclic;
	
	(ii) char $k = p$;
	
	(iii) $\dim X < p - 1$.

\noindent Then $X(k)^G = \varnothing$ if and only if the Euler characteristic $\chi(X, \cl{F})$ of every $G$-equivariant coherent $\oo_X$-module $\cl{F}$ is divisible by $p$.
\end{thm}

Observe that the structure sheaf $\oo_X$ is $G$-equivariant. This leads to our main theorem:

\begin{thm} \label{main}
	Let $X$ be a rationally connected variety of dimension $n$. If $G \subset \bir(X)$ is a finite $p$-subgroup and $p > n +1$, then $G$ is abelian and the rank of $G$ is at most $n$.
\end{thm}
\begin{proof}
	Passing to a smooth regularization of $G$ \cite[Lemma-Definition 3.1]{ps2}, we may assume that $X$ is nonsingular and projective, and the group $G$ acts biregularly on $X$. It is well known that for a nonsingular projective rationally connected variety $X$, one has $\chi(\oo_X) = 1$. Therefore by Theorem \ref{fixed} $(iii)$, $G$ has a fixed point $x \in X$ as long as $p > n+1$. The action of $G$ on the Zariski tangent space $T_{x,X} \simeq k^n$ is faithful so that $G$ embeds into the general linear group $\gl(n, k)$. Using \cite[Lemma 3.3]{ps4}, we conclude that $G$ is abelian of rank $\leq n$.
\end{proof}

\begin{exam} \label{example}
	(1) The rank $r$ of an elementary $p$-subgroup of $\bir(X)$ can be greater than $n$ if $p \leq n +1$. For rational surfaces, it is known that $r \leq 4$ if $p = 2$, and $r \leq 3$ if $p = 3$. Both of these upper bounds are sharp, see \cite{be}. For rationally connected threefolds, it is known that $r \leq 6$ if $p = 2$ \cite[Theorem 1.2]{p1}, which is sharp. If $p = 3$, the product of the Fermat cubic surface with $\BP^1$ gives an example showing that $(\Z/3\Z)^4$ embeds into $\mathrm{Cr}_3(k)$. Therefore we can take $L(3) = 3$, which is sharp in dimension $3$.
	
	We also know that $r \leq 5$ when $p = 3$ \cite[Theorem 1.2]{p2}, but it is not clear whether $r = 5$ can be reached.

	(2) The lower bound $p > n+1$ in Theorem \ref{main} is \emph{sharp} when $n+1$ is a \emph{prime}: the Fermat hypersurface $X \subset \BP^{n+1}$ of degree $n+1$:
	$$
		x_0^{n+1} + x_1^{n+1} + \cdots + x_{n+1}^{n+1} = 0.
	$$
	is Fano and hence rationally connected, it admits a faithful action of $(\Z/(n+1)\Z)^{n+1}$, an abelian group of rank $n+1$.
\end{exam}

\textbf{Acknowledgement.} This note was written during the author's visit to Laboratory of Algebraic Geometry, HSE University in Dec 2019. The author would like to thank Constantin Shramov for his invitation and valuable suggestions on Example \ref{example}. The author also thanks Yuri Prokhorov for useful discussion.

%%%%%%%%%%%%%%%%%%%

\noindent Department of Mathematical Sciences,\\
\noindent Xi'an Jiaotong-Liverpool University,\\
\noindent No.111 Ren'ai Road, SIP, Suzhou, Jiangsu Province, China\\
\noindent e-mail: jinsong.xu@xjtlu.edu.cn
\end{document}